\documentclass[a4paper,12pt,reqno]{amsart}

\usepackage{graphicx,amssymb,datetime,float,MnSymbol,currfile,tikz}
\usepackage[round]{natbib}
\usetikzlibrary{arrows}
\usetikzlibrary{decorations.markings}

\newtheorem{theorem}{Theorem}
\newtheorem{lemma}[theorem]{Lemma}
\newtheorem{corollary}[theorem]{Corollary}

\def\ci{\!\perp\!}
\def\nci{\!\not\perp\!}
\def\ra{\rightarrow}

\def\la{\leftarrow}

\def\aa{\leftrightarrow}

\def\ao{\leftarrow\!\!\!\!\!\multimap}

\def\oa{\mathrel{\reflectbox{\ensuremath{\ao}}}}

\def\oo{\mathrel{\reflectbox{\ensuremath{\multimap}}}\!\!\!\!\!\multimap}

\newcommand{\comments}[1]{}

\tikzset{tt/.style={decoration={
  markings,
  mark=at position .485 with {\arrow{>}},
  mark=at position .515 with {\arrow{<}}},postaction={decorate}}}

\begin{document}

\title[]{Towards Conditional Path Analysis$^*$}\thanks{$^*$This work extends our results in \citet{Penna2020} from singly-connected path diagrams to a superclass thereof.}

\author[]{Jose M. Pe\~{n}a\\
IDA, Link\"oping University, Sweden\\
jose.m.pena@liu.se}


\begin{abstract}
We extend path analysis by giving sufficient conditions for computing the partial covariance of two random variables from their covariance. This is specifically done by correcting the covariance with the product of some partial variance ratios. As a result, the partial covariance retains the covariance's salient feature of factorizing over the edges in the paths between the two variables of interest.
\end{abstract}

\maketitle

\section{Introduction}

To ease interpretation, linear structural equation models are typically represented as path diagrams: Nodes represent random variables, directed edges represent direct causal relationships, and bidirected edges represent confounding, i.e. correlation between error terms. Moreover, each directed edge is annotated with the corresponding coefficient in the linear structural equation model, a.k.a. path coefficient. Likewise, each bidirected edge is annotated with the corresponding error covariance. A path diagram also brings in computational benefits. For instance, it is known that the covariance $\sigma_{XY}$ of two standardized random variables $X$ and $Y$ can be determined from the path diagram. Specifically, $\sigma_{XY}$ can be expressed as the sum for every $\emptyset$-open path from $X$ to $Y$ of the product of path coefficients and error covariances for the edges in the path \citep{Wright1921,Pearl2009}. For non-standardized variables, one has to multiply the product associated to each path with the variance of the root variable in the path, i.e. the variable with no incoming edges. A path can have no root variables ($X \aa Z \ra \cdots \ra Y$ or $X \la \cdots \la Z \aa W \ra \cdots \ra Y$) or one root variable ($X \ra \cdots \ra Y$ or $X \la \cdots \la Z \ra \cdots \ra Y$).

In this work, we develop a similar factorization for the partial covariance $\sigma_{XY \cdot Z}$ for certain path diagrams. More specifically, we give sufficient conditions for computing $\sigma_{XY \cdot Z}$ by correcting $\sigma_{XY}$ with the product of some partial variance ratios. This implies that $\sigma_{XY \cdot Z}$ factorizes over the edges of the paths from $X$ to $Y$, much in the same way as $\sigma_{XY}$ does. Moreover, we use the factorization to show that Simpson's paradox cannot occur in certain path diagrams.

We conclude this introduction by recalling some related works. \citet{ChaudhuriandRichardson2003} and \citet{Chaudhuri2005} identify sufficient conditional independencies for ordering squared partial correlation coefficients in singly-connected path diagrams. \citet{Chaudhuri2014} extends these results to general Gaussian random vectors. \citet{ChaudhuriandTan2010} report similar general results for absolute values of partial regression coefficients. Finally, \citet{Ong2014} proves similar results for (signed) partial covariances, correlation coefficients and regression coefficients for singly-connected path diagrams and general Gaussian random vectors. It should be noted that these works consider paths with colliders and we do not. However, none of these works develops a factorization of the measure of association under study, as we do for the partial covariance. Moreover, the path diagrams that we consider in this work are a superclass of singly-connected diagrams. Therefore, our results about the impossibility of Simpson's paradox subsume the results by \citet{Pearl2014}, who identified three singly-connected path diagrams that do not lead to the paradox.

The rest of this work is structured as follows. Section \ref{sec:condpath} presents our factorization of partial covariances, and the conditions under which it is valid. The section also features some illustrative examples. Section \ref{sec:Simpson} shows that Simpson's paradox cannot occur under the conditions in the previous section. Section \ref{sec:discussion} closes with some discussion.

\section{Conditional Path Analysis}\label{sec:condpath}

We start by recalling the separation criterion for path diagrams. Given a path $\pi_{X:Y}$ from a node $X$ to a node $Y$ in a path diagram, a node $C$ is a collider in $\pi_{X:Y}$ if $A \oa C \ao B$ is a subpath of $\pi_{X:Y}$, where $\oa$ means $\ra$ or $\aa$. Given a set of nodes $Z$, $\pi_{X:Y}$ is said to be $Z$-open if
\begin{itemize}
\item every collider in $\pi_{X:Y}$ is in $Z$ or has some descendant in $Z$, and
\item every non-collider in $\pi_{X:Y}$ is outside $Z$.
\end{itemize}
Whereas all the nodes in a path must be different, the nodes in a route do not need to be so. Given a route $\rho_{X:Y}$ from a node $X$ to a node $Y$ in a path diagram, a node $C$ is a collider in $\rho_{X:Y}$ if $A \oa C \ao B$ is a subroute of $\rho_{X:Y}$. Note that $A$ and $B$ may be the same node. Given a set of nodes $Z$, $\rho_{X:Y}$ is said to be $Z$-open if
\begin{itemize}
\item every collider in $\rho_{X:Y}$ is in $Z$, and
\item ever non-collider in $\rho_{X:Y}$ is outside $Z$.
\end{itemize}
Note that there is a $Z$-open route from $X$ to $Y$ if and only if there is a $Z$-open path from $X$ to $Y$ (see Lemma \ref{lem:pathroute} in the appendix). We use $X \nci Y | Z$ and $X \ci Y | Z$ to denote, respectively, the existence or not of such a path. We say that $X$ and $Y$ are $Z$-connected if $X \nci Y | Z$.

If $X \ci Y | Z$, then one can readily conclude that $\sigma_{XY \cdot Z}=0$. If on the other hand $X \nci Y | Z$, then one may think that $\sigma_{XY \cdot Z}$ can be obtained by first applying path analysis to obtain an expression for $\sigma_{XY}$ and, then, modifying this expression by replacing (co)variances with partial (co)variances given $Z$. However, this is incorrect as the following example shows. Consider the path diagram to the left in Figure \ref{fig:counterexample}, which corresponds to the following linear structural equation model:
\begin{align*}
X &= \epsilon_X\\
Y &= \alpha X + \epsilon_Y\\
Z &= \delta Y + \epsilon_Z.
\end{align*}
Consider representing the error terms explicitly in the diagram, which results in the path diagram to the right in Figure \ref{fig:counterexample}. Then,
\[
\sigma_{XY} = cov(X, \alpha X + \epsilon_Y) = \alpha \sigma^2_X + cov(X,\epsilon_Y) = \alpha \sigma^2_X
\]
where the last equality follows from the fact that $cov(X,\epsilon_Y)=0$ since $X \ci \epsilon_Y | \emptyset$. However, 
\[
\sigma_{XY \cdot Z} = cov(X, \alpha X + \epsilon_Y | Z) = \alpha \sigma^2_{X \cdot Z} + cov(X,\epsilon_Y | Z) \neq \alpha \sigma^2_{X \cdot Z}
\]
where the last inequality follows from the fact that $cov(X,\epsilon_Y | Z) \neq 0$ in general, since $X \ci \epsilon_Y | Z$ does not hold. The theorems below show that we can obtain $\sigma_{XY \cdot Z}$ by correcting $\sigma_{XY}$ with a partial variance ratio as follows:
\[
\sigma_{XY \cdot Z} = \sigma_{XY} \frac{\sigma^2_{Y \cdot Z}}{\sigma^2_Y}.
\]

\begin{figure}
\begin{tabular}{c|c}
\begin{tabular}{c}
\begin{tikzpicture}[inner sep=1mm]
\node at (0,0) (X) {$X$};
\node at (1.5,-1.2) (Z) {$Z$};
\node at (1.5,0) (Y) {$Y$};
\path[->] (X) edge node[above] {$\alpha$} (Y);
\path[->] (Y) edge node[right] {$\delta$} (Z);
\end{tikzpicture}
\end{tabular}
&
\begin{tabular}{c}
\begin{tikzpicture}[inner sep=1mm]
\node at (0,0) (X) {$X$};
\node at (0,1.2) (eX) {$\epsilon_X$};
\node at (1.5,-1.2) (Z) {$Z$};
\node at (3,-1.2) (eZ) {$\epsilon_Z$};
\node at (1.5,0) (Y) {$Y$};
\node at (3,0) (eY) {$\epsilon_Y$};
\path[->] (X) edge node[above] {$\alpha$} (Y);
\path[->] (Y) edge node[right] {$\delta$} (Z);
\path[->] (eX) edge node[right] {$1$} (X);
\path[->] (eY) edge node[above] {$1$} (Y);
\path[->] (eZ) edge node[above] {$1$} (Z);
\end{tikzpicture}
\end{tabular}
\end{tabular}\caption{Path diagram illustrating that computing $\sigma_{XY \cdot Z}$ is non-trivial.}\label{fig:counterexample}
\end{figure}
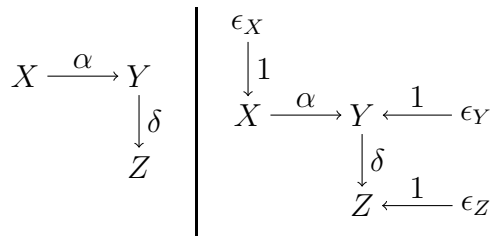

\subsection{Main Contribution}\label{sec:main}

The two theorems in this section are the main contribution of this work. Before presenting them, we need to introduce some notation. The parents of a node $X$ are $Pa(X) = \{Y | Y \ra X \}$. The children of $X$ are $Ch(X) = \{Y | X \ra Y \}$. The spouses of $X$ are $Sp(X) = \{Y | X \aa Y \}$. Moreover, we define the operation of conditioning a path diagram on a node $A$ as replacing every edge $A \ra B$ with an edge $A_B \ra B$, where $A_B$ is a new node. Note that $A$ is not removed. In terms of the associated system of linear equations, this implies (i) adding a new equation $A_B = \epsilon_{A_B}$ where $\epsilon_{A_B}$ is normally distributed with arbitrary mean and variance, and (ii) replacing every equation $B = \alpha^T (A, Pa(B) \setminus A) + \epsilon_B$ with an equation $B = \alpha^T (A_B, Pa(B) \setminus A) + \epsilon_B$. Note that, after conditioning, we have that $Ch(A)=\emptyset$ whereas $Pa(A_B) \cup Sp(A_B)=\emptyset$ and $Ch(A_B)=B$. See Figure \ref{fig:illustration} for an illustration. Let $V$ denote all the nodes in the path diagram at hand, and consider the distribution $p(V \setminus A, A=a)$ defined by the system of equations before conditioning on $A$. This is the unnormalized conditional distribution of $V \setminus A$ given $A$. Let $A'$ denote the new nodes created by conditioning on $A$, and consider the distribution $p(V \setminus A, A=a, A'=a)$ defined by the system of equations after conditioning on $A$. This is the unnormalized conditional distribution of $V \setminus A$ given $A \cup A'$. Note that both unnormalized conditional distributions coincide for all $a$, i.e. $p(V \setminus A = x, A=a) = p(V \setminus A = x, A=a, A'=a)$ for all $x$ and $a$. Thus, their normalized versions coincide as well. So, computing partial covariances in either of them gives the same result, since partial covariances coincide with conditional covariances for Gaussian random vectors. In other words, the partial covariance $\sigma_{XY \cdot A}$ in the original path diagram is equal to $\sigma_{XY \cdot A A'}$ in the conditional diagram. Finally, we define conditioning on a set of nodes $S$ as conditioning on each node in $S$. By the previous reasoning, the partial covariance $\sigma_{XY \cdot S}$ in the original path diagram is equal to $\sigma_{XY \cdot S S'}$ in the conditional diagram, where $S'$ denotes the new nodes created by conditioning on $S$. The following theorems show how to compute the latter. See the appendix for the proofs. We illustrate the theorems through some examples in the next section.

\begin{figure}
\begin{tabular}{c|c}
\begin{tabular}{c}
\begin{tikzpicture}[inner sep=1mm]
\node at (0,0) (W) {$A$};
\node at (-1,1) (X1) {};
\node at (1,1) (X2) {};
\node at (-1.3,0) (X3) {};
\node at (1.3,0) (X4) {};
\node at (-1,-1) (X5) {$B$};
\node at (1,-1) (X6) {$C$};
\path[->] (X1) edge (W);
\path[->] (X2) edge (W);
\path[<->] (X3) edge (W);
\path[<->] (X4) edge (W);
\path[<-] (X5) edge (W);
\path[<-] (X6) edge (W);
\end{tikzpicture}
\end{tabular}
&
\begin{tabular}{c}
\begin{tikzpicture}[inner sep=1mm]
\node at (0,0) (W) {$A$};
\node at (-0.5,-1) (W1) {$A_B$};
\node at (0.5,-1) (W2) {$A_C$};
\node at (-1,1) (X1) {};
\node at (1,1) (X2) {};
\node at (-1.3,0) (X3) {};
\node at (1.3,0) (X4) {};
\node at (-1.5,-2) (X5) {$B$};
\node at (1.5,-2) (X6) {$C$};
\path[->] (X1) edge (W);
\path[->] (X2) edge (W);
\path[<->] (X3) edge (W);
\path[<->] (X4) edge (W);
\path[<-] (X5) edge (W1);
\path[<-] (X6) edge (W2);
\end{tikzpicture}
\end{tabular}
\end{tabular}\caption{Conditioning the path diagram to the left on the node $A$ results in the diagram to the right.}\label{fig:illustration}
\end{figure}
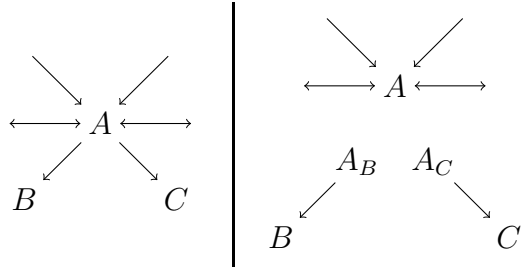

\begin{theorem}\label{the:condpathroot}
Consider a path diagram conditioned on a set of nodes $S$. Let $Z = S \cup S'$. Let $\Pi_{X:Y}$ denote all the $Z$-open paths from $X$ to $Y$. Suppose that no path in $\Pi_{X:Y}$ has colliders. Suppose that all the paths in $\Pi_{X:Y}$ have a subpath $X_m \la \cdots \la X_2 \la X_1 \ra X_{m+1} \ra \cdots \ra X_{m+n}$ or $X_1 = X \ra X_{2} \ra \cdots \ra X_{m+n}$. Suppose that there is no $Z$-open route $X_i \ra A \oo \cdots \oo B \oa X_i$ with $i>1$. Moreover, let $Z_i^i=Z_i \cup Z^i$ and $Z_{1:a}^{1:b} = Z_1 \cup \cdots \cup Z_a \cup Z^1 \cup \cdots \cup Z^b$ where
\begin{itemize}
\item $Z^i = \{W_1, W_2, \ldots\}$ is a subset of $Z \setminus Z^{1:i-1}_{1:i-1}$ such that each $W_j$ is $(Z^{1:i-1}_{1:i-1} \cup W_{1:j-1})$-connected to $X_i$ through $Pa(X_i) \cup Sp(X_i)$ by a path that does not contain any node that is in some path in $\Pi_{X:Y}$ except $X_i$, and
\item $Z_i = \{W_1, W_2, \ldots\}$ is a subset of $Z \setminus Z^{1:i}_{1:i-1}$ such that each $W_j$ is $(Z^{1:i}_{1:i-1} \cup W_{1:j-1})$-connected to $X_i$ through $Ch(X_i)$ by a path that does not contain any node that is in some path in $\Pi_{X:Y}$ except $X_i$.
\end{itemize}
Finally, let $Z \setminus Z^{1:m+n}_{1:m+n} = \{W_1, W_2, \ldots\}$ and suppose that $X \ci W_j | Z^{1:m+n}_{1:m+n} \cup W_{1:j-1}$ or $Y \ci W_j | Z^{1:m+n}_{1:m+n} \cup W_{1:j-1}$. Then,
\[
\sigma_{X Y \cdot Z} = \sigma_{X Y} \frac{\sigma^2_{X_1 \cdot Z_1^1}}{\sigma^2_{X_1}} \prod_{i=2}^{m+n} \frac{\sigma^2_{X_i \cdot Z_{1:i}^{1:i}}}{\sigma^2_{X_i \cdot Z_{1:i-1}^{1:i}}}.
\]
\end{theorem}

\begin{theorem}\label{the:condpathnonroot}
Consider the same assumptions as in Theorem \ref{the:condpathroot} with the only exception that all the paths in $\Pi_{X:Y}$ have now a subpath $X_m \la \cdots \la X_2 \la X_1 \aa X_{m+1} \ra \cdots \ra X_{m+n}$ or $X_1 \aa X_{2} \ra \cdots \ra X_{m+n}$ or $\oa X_1 \ra \cdots \ra X_{m+n}$.\footnote{In the third subpath type, the predecessor of $X_1$ does not have to be the same in every path in $\Pi_{X:Y}$. It just has to reach $X_1$ through an edge $\ra$ or $\aa$ in every path in $\Pi_{X:Y}$.\label{foo:subpath}} Then,
\[
\sigma_{X Y \cdot Z} = \sigma_{X Y} \prod_{i=1}^{m+n} \frac{\sigma^2_{X_i \cdot Z_{1:i}^{1:i}}}{\sigma^2_{X_i \cdot Z_{1:i-1}^{1:i}}}.
\]
\end{theorem}

Note that the numerator and the denominator of each partial variance ratio in the theorems above coincide except for $Z_i$ that is only in the numerator. Then, each ratio can be interpreted as a deflation factor ($\leq 1$) that accounts for the additional reduction of the partial variance of $X_i$ when conditioning on $Z_{i}$. Note that $\sigma_{X Y \cdot Z}$ inherits from $\sigma_{X Y}$ the salient feature of factorizing over the edges in the paths in $\Pi_{X:Y}$. Note also that the theorems always apply when the original path diagram is singly-connected. Finally, note that conditioning does not change the sign of the covariance, i.e. $sign(\sigma_{X Y \cdot Z}) = sign(\sigma_{X Y})$. This implies that if all the paths in $\Pi_{XY}$ are of the form $X \ra \cdots \ra Y$, then conditioning does not change the sign of the regression coefficient of $Y$ on $X$ and, thus, of the causal effect of $X$ on $Y$. The following corollary is immediate.

\begin{corollary}\label{cor:samesign}
Suppose that two sets of nodes $S_1$ and $S_2$ satisfy the assumptions in Theorems \ref{the:condpathroot} or \ref{the:condpathnonroot}. Then, $sign(\sigma_{XY}) = sign(\sigma_{XY \cdot S_1})=sign(\sigma_{XY \cdot S_2})$.
\end{corollary}

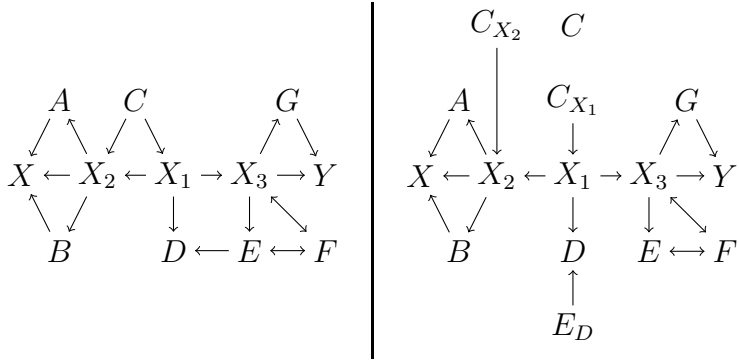
\begin{figure}
\begin{tabular}{c|c}
\begin{tabular}{c}
\begin{tikzpicture}[inner sep=1mm]
\node at (0,0) (X) {$X$};
\node at (1,0) (X1) {$X_2$};
\node at (2,0) (X2) {$X_1$};
\node at (3,0) (X3) {$X_3$};
\node at (4,0) (Y) {$Y$};
\node at (3.5,1) (G) {$G$};
\node at (0.5,1) (A) {$A$};
\node at (0.5,-1) (B) {$B$};
\node at (4,-1) (C) {$F$};
\node at (1.5,1) (S1) {$C$};
\node at (2,-1) (S2) {$D$};
\node at (3,-1) (S3) {$E$};
\path[<-] (X) edge (X1);
\path[<-] (X1) edge (X2);
\path[->] (X2) edge (X3);
\path[->] (X3) edge (Y);
\path[<-] (X) edge (A);
\path[<-] (A) edge (X1);
\path[<-] (B) edge (X1);
\path[->] (B) edge (X);
\path[<->] (X3) edge (C);
\path[<->] (C) edge (S3);
\path[->] (S1) edge (X1);
\path[->] (S1) edge (X2);
\path[->] (X2) edge (S2);
\path[->] (S3) edge (S2);
\path[<-] (S3) edge (X3);
\path[->] (X3) edge (G);
\path[->] (G) edge (Y);
\end{tikzpicture}
\end{tabular}
&
\begin{tabular}{c}
\begin{tikzpicture}[inner sep=1mm]
\node at (0,0) (X) {$X$};
\node at (1,0) (X1) {$X_2$};
\node at (2,0) (X2) {$X_1$};
\node at (3,0) (X3) {$X_3$};
\node at (4,0) (Y) {$Y$};
\node at (3.5,1) (G) {$G$};
\node at (0.5,1) (A) {$A$};
\node at (0.5,-1) (B) {$B$};
\node at (4,-1) (C) {$F$};
\node at (2,2) (D) {$C$};
\node at (1,2) (S1X1) {$C_{X_2}$};
\node at (2,1) (S1X2) {$C_{X_1}$};
\node at (2,-1) (S2) {$D$};
\node at (3,-1) (S3) {$E$};
\node at (2,-2) (S3S2) {$E_{D}$};
\path[<-] (X) edge (X1);
\path[<-] (X1) edge (X2);
\path[->] (X2) edge (X3);
\path[->] (X3) edge (Y);
\path[<-] (X) edge (A);
\path[<-] (A) edge (X1);
\path[<-] (B) edge (X1);
\path[->] (B) edge (X);
\path[<->] (X3) edge (C);
\path[<->] (C) edge (S3);
\path[->] (S1X1) edge (X1);
\path[->] (S1X2) edge (X2);
\path[->] (X2) edge (S2);
\path[->] (S3S2) edge (S2);
\path[<-] (S3) edge (X3);
\path[->] (X3) edge (G);
\path[->] (G) edge (Y);
\end{tikzpicture}
\end{tabular}
\end{tabular}\caption{Left: Path diagram where Theorem \ref{the:condpathroot} can be applied to compute $\sigma_{XY \cdot C D E}$. Right: The path diagram to the left conditioned on $\{C, D, E\}$.}\label{fig:example}
\end{figure}

\subsection{Examples}

We illustrate Theorem \ref{the:condpathroot} with the following example. Consider the path diagram to the left in Figure \ref{fig:example}. Say that we want to compute $\sigma_{XY \cdot S}$ with $S=\{C,D,E\}$. The path diagram conditioned on $S$ can be seen to the right in Figure \ref{fig:example}. Then, $S'=\{C_{X_1}, C_{X_2}, E_D\}$ and $Z = S \cup S' = \{C,D,E,C_{X_1}, C_{X_2}, E_D\}$. As discussed before, $\sigma_{XY \cdot S}$ in the original diagram coincides with $\sigma_{X Y \cdot Z}$ in the conditional diagram. Now, note that $Z^1=\{C_{X_1}\}$, $Z_1=\{D,E_D\}$, $Z^2=\{C_{X_2}\}$, $Z_2=\emptyset$, $Z^3=\emptyset$, and $Z_3=\{E\}$. Then, Theorem \ref{the:condpathroot} gives
\[
\sigma_{X Y \cdot Z} = \sigma_{X Y} \frac{\sigma^2_{X_1 \cdot C_{X_1} D E_D}}{\sigma^2_{X_1}} \frac{\sigma^2_{X_2 \cdot C_{X_1} D E_D C_{X_2}}}{\sigma^2_{X_2 \cdot C_{X_1} D E_D C_{X_2}}} \frac{\sigma^2_{X_3 \cdot C_{X_1} D E_D C_{X_2} E}}{\sigma^2_{X_3 \cdot C_{X_1} D E_D C_{X_2}}}.
\]

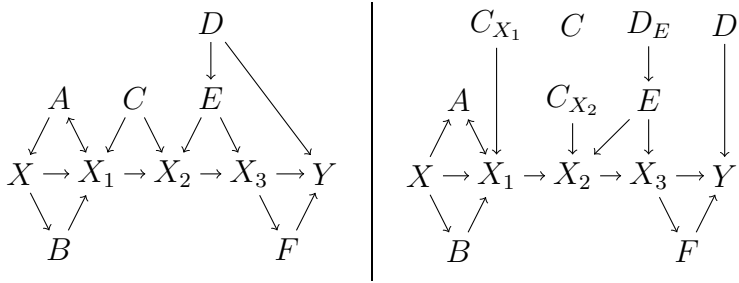
\begin{figure}
\begin{tabular}{c|c}
\begin{tabular}{c}
\begin{tikzpicture}[inner sep=1mm]
\node at (0,0) (X) {$X$};
\node at (1,0) (X1) {$X_1$};
\node at (2,0) (X2) {$X_2$};
\node at (3,0) (X3) {$X_3$};
\node at (4,0) (Y) {$Y$};
\node at (0.5,1) (A) {$A$};
\node at (0.5,-1) (B) {$B$};
\node at (1.5,1) (C) {$C$};
\node at (2.5,2) (D) {$D$};
\node at (2.5,1) (E) {$E$};
\node at (3.5,-1) (F) {$F$};
\path[->] (X) edge (X1);
\path[->] (X1) edge (X2);
\path[->] (X2) edge (X3);
\path[->] (X3) edge (Y);
\path[<-] (X) edge (A);
\path[<->] (A) edge (X1);
\path[->] (B) edge (X1);
\path[<-] (B) edge (X);
\path[->] (C) edge (X1);
\path[->] (C) edge (X2);
\path[->] (E) edge (X2);
\path[->] (E) edge (X3);
\path[->] (D) edge (E);
\path[->] (D) edge (Y);
\path[->] (X3) edge (F);
\path[->] (F) edge (Y);
\end{tikzpicture}
\end{tabular}
&
\begin{tabular}{c}
\begin{tikzpicture}[inner sep=1mm]
\node at (0,0) (X) {$X$};
\node at (1,0) (X1) {$X_1$};
\node at (2,0) (X2) {$X_2$};
\node at (3,0) (X3) {$X_3$};
\node at (4,0) (Y) {$Y$};
\node at (0.5,1) (A) {$A$};
\node at (0.5,-1) (B) {$B$};
\node at (2,2) (D) {$C$};
\node at (1,2) (S1X1) {$C_{X_1}$};
\node at (2,1) (S1X2) {$C_{X_2}$};
\node at (4,2) (D) {$D$};
\node at (3,2) (DE) {$D_E$};
\node at (3,1) (E) {$E$};
\node at (3.5,-1) (F) {$F$};
\path[->] (X) edge (X1);
\path[->] (X1) edge (X2);
\path[->] (X2) edge (X3);
\path[->] (X3) edge (Y);
\path[->] (X) edge (A);
\path[<->] (A) edge (X1);
\path[->] (B) edge (X1);
\path[<-] (B) edge (X);
\path[->] (S1X1) edge (X1);
\path[->] (S1X2) edge (X2);
\path[->] (E) edge (X2);
\path[->] (E) edge (X3);
\path[->] (D) edge (Y);
\path[->] (DE) edge (E);
\path[->] (X3) edge (F);
\path[->] (F) edge (Y);
\end{tikzpicture}
\end{tabular}
\end{tabular}\caption{Left: Path diagram where Theorem \ref{the:condpathnonroot} can be applied to compute $\sigma_{XY \cdot C D}$. Right: The path diagram to the left conditioned on $\{C, D\}$.}\label{fig:example2}
\end{figure}

We also illustrate Theorem \ref{the:condpathnonroot} with an example. Consider the path diagram to the left in Figure \ref{fig:example2}. Say that we want to compute $\sigma_{XY \cdot S}$ with $S=\{C,D\}$. The path diagram conditioned on $S$ can be seen to the right in Figure \ref{fig:example2}. Then, $S'=\{C_{X_1}, C_{X_2}, D_E\}$ and $Z = S \cup S' = \{C,D,C_{X_1}, C_{X_2}, D_E\}$. As discussed before, $\sigma_{XY \cdot S}$ in the original diagram coincides with $\sigma_{X Y \cdot Z}$ in the conditional diagram. Now, note that $Z^1=\{C_{X_1}\}$, $Z^2=\{C_{X_2}, D_E\}$, and $Z_1=Z_2=Z_3=Z^3=\emptyset$. Then, Theorem \ref{the:condpathnonroot} gives
\[
\sigma_{X Y \cdot Z} = \sigma_{X Y} \frac{\sigma^2_{X_1 \cdot C_{X_1}}}{\sigma^2_{X_1 \cdot C_{X_1}}} \frac{\sigma^2_{X_2 \cdot C_{X_1} C_{X_2} D_E}}{\sigma^2_{X_2 \cdot C_{X_1} C_{X_2} D_E}} \frac{\sigma^2_{X_3 \cdot C_{X_1} C_{X_2} D_E}}{\sigma^2_{X_3 \cdot C_{X_1} C_{X_2} D_E}}.
\]

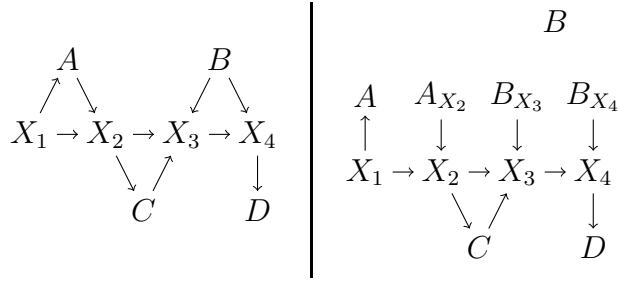
\begin{figure}
\begin{tabular}{c|c}
\begin{tabular}{c}
\begin{tikzpicture}[inner sep=1mm]
\node at (0,0) (X) {$X_1$};
\node at (1,0) (X1) {$X_2$};
\node at (2,0) (X2) {$X_3$};
\node at (3,0) (Y) {$X_4$};
\node at (0.5,1) (A) {$A$};
\node at (1.5,-1) (B) {$C$};
\node at (2.5,1) (C) {$B$};
\node at (3,-1) (D) {$D$};
\path[->] (X) edge (X1);
\path[->] (X1) edge (X2);
\path[->] (X2) edge (Y);
\path[<-] (X1) edge (A);
\path[<-] (A) edge (X);
\path[->] (C) edge (X2);
\path[->] (C) edge (Y);
\path[->] (X1) edge (B);
\path[->] (B) edge (X2);
\path[->] (Y) edge (D);
\end{tikzpicture}
\end{tabular}
&
\begin{tabular}{c}
\begin{tikzpicture}[inner sep=1mm]
\node at (0,0) (X) {$X_1$};
\node at (1,0) (X1) {$X_2$};
\node at (2,0) (X2) {$X_3$};
\node at (3,0) (Y) {$X_4$};
\node at (0,1) (A) {$A$};
\node at (1,1) (AX2) {$A_{X_2}$};
\node at (1.5,-1) (B) {$C$};
\node at (2.5,2) (C) {$B$};
\node at (2,1) (CX2) {$B_{X_3}$};
\node at (3,1) (CY) {$B_{X_4}$};
\node at (3,-1) (D) {$D$};
\path[->] (X) edge (X1);
\path[->] (X1) edge (X2);
\path[->] (X2) edge (Y);
\path[<-] (X1) edge (AX2);
\path[<-] (A) edge (X);
\path[->] (X1) edge (B);
\path[->] (B) edge (X2);
\path[->] (CX2) edge (X2);
\path[->] (CY) edge (Y);
\path[->] (Y) edge (D);
\end{tikzpicture}
\end{tabular}
\end{tabular}\caption{Left: Path diagram where Theorems \ref{the:condpathroot} and \ref{the:condpathnonroot} can be combined to compute $\sigma_{X_1 X_4 \cdot A B D}$. Left: The path diagram to the left conditioned on $\{A, B, D\}$.}\label{fig:example3}
\end{figure}

Theorems \ref{the:condpathroot} and \ref{the:condpathnonroot} can be extended to when all the paths from $X$ to $Y$ in the conditional path diagram share more than one subpath. For instance, consider the path diagram to the left in Figure \ref{fig:example3}. Say that we want to compute $\sigma_{X_1 X_4 \cdot S}$ with $S=\{A, B, D\}$. The path diagram conditioned on $S$ can be seen to the right in Figure \ref{fig:example3}. Then, $S'=\{A_{X_2}, B_{X_3}, B_{X_4}\}$ and $Z = S \cup S' = \{A, B, D, A_{X_2}, B_{X_3}, B_{X_4}\}$. As discussed before, $\sigma_{X_1 X_4 \cdot S}$ in the original diagram coincides with $\sigma_{X_1 X_4 \cdot Z}$ in the conditional diagram. Now, note that $Z_1=\{A\}$, $Z^2=\{A_{X_2}\}$, $Z^3=\{B_{X_3}\}$, $Z^4=\{B_{X_4}\}$, $Z_4=\{D\}$, and $Z^1=Z_2=Z_3=\emptyset$. Note also that the conditional diagram has two $Z$-open paths from $X_1$ to $X_4$, which share two subpaths: $X_1 \ra X_2$ and $\ra X_3 \ra X_4$. Therefore, neither Theorem \ref{the:condpathroot} nor \ref{the:condpathnonroot} applies. However, applying Theorem \ref{the:condpathroot} followed by Theorem \ref{the:condpathnonroot} gives
\begin{equation}\label{eq:example3}
\sigma_{X_1 X_4 \cdot Z} = \sigma_{X_1 X_4} \frac{\sigma^2_{X_1 \cdot A}}{\sigma^2_{X_1}} \frac{\sigma^2_{X_2 \cdot A A_{X_2}}}{\sigma^2_{X_2 \cdot A A_{X_2}}} \frac{\sigma^2_{X_3 \cdot A A_{X_2} B_{X_3}}}{\sigma^2_{X_3 \cdot A A_{X_2} B_{X_3}}} \frac{\sigma^2_{X_4 \cdot A A_{X_2} B_{X_3} B_{X_4} D}}{\sigma^2_{X_4 \cdot A A_{X_2} B_{X_3} B_{X_4}}}.
\end{equation}
The proof of correctness of the previous expression is simply a concatenation of the proofs of Theorems \ref{the:condpathroot} and \ref{the:condpathnonroot}. We omit the details. An alternative way of answering the previous query is by first absorbing the subpath $X_2 \ra C \ra X_3$ into the subpath $X_2 \ra X_3$. Now, there is only one shared subpath in the path diagram conditioned on $S$, namely $X_1 \ra X_2 \ra X_3 \ra X_4$. Then, Theorem \ref{the:condpathroot} gives Equation \ref{eq:example3}. This absorption trick is always possible when there are several shared subpaths. We omit the details.

\section{Simpson's Paradox}\label{sec:Simpson}

In this section, we use Theorems \ref{the:condpathroot} and \ref{the:condpathnonroot} to conclude that Simpson's paradox cannot occur in certain path diagrams. For path diagrams, Simpson's paradox can be described as the reversal of the sign of the regression coefficient of a random variable $Y$ on a second variable $X$ upon conditioning on a set of variables $S$. \citet[Section 3.1]{Pearl2013} shows that the paradox can well occur for the path diagram $X \ra Y \la S \ra X$. Note the diagram is not singly-connected. \citet[Section 2.2]{Pearl2014} argues that the paradox does not occur for the singly-connected path diagrams $S \la X \ra Y$, $S \ra X \ra Y$, and $X \ra Y \la S$, because the association between $X$ and $Y$ is collapsible over $S$. However, the correctness of this statement depends on the definition of association. To see it, recall from \citet[Definition 6.5.1]{Pearl2009} that given a functional $g(p(x,y))$ that measures the association between two random variables $Y$ and $X$ in $p(x,y)$, we say that $g$ is collapsible over a variable $S$ if
\[
E_s[g(p(x,y|s))] = g(p(x,y)).
\]
If we now consider the diagram $S \la X \ra Y$ and let $g$ be the covariance between $Y$ and $X$, then collapsibility does not hold since
\begin{align*}
E_s[g(p(x,y|s))] = E_s[cov(X,Y|S=s)] &= cov(X,Y|S)\\
& = \sigma_{XY \cdot S} \neq \sigma_{XY} = g(p(x,y))
\end{align*}
where the second equality follows from the fact that the conditional covariance does not depend on the value of the conditioning set, and the inequality follows from Theorem \ref{the:condpathroot}.\footnote{For $S \la X \ra Y$ and $S \ra X \ra Y$, Theorem \ref{the:condpathroot} and $X \nci S | \emptyset$ give $\sigma_{XY \cdot S} = \sigma_{XY} ( \sigma^2_{X \cdot S} / \sigma^2_X ) ( \sigma^2_{Y \cdot S} / \sigma^2_{Y \cdot S} ) \neq \sigma_{XY}$.} Similarly for the diagram $S \ra X \ra Y$. For the diagram $X \ra Y \la S$, on the other hand, Theorem \ref{the:condpathroot} implies collapsibility.\footnote{For $X \ra Y \la S$, Theorem \ref{the:condpathroot} gives $\sigma_{XY \cdot S} = \sigma_{XY} ( \sigma^2_{X} / \sigma^2_{X} ) ( \sigma^2_{Y \cdot S} / \sigma^2_{Y \cdot S} ) = \sigma_{XY}$.} If we instead let $g$ be the regression coefficient of $Y$ on $X$, then collapsibility follows from Theorem \ref{the:condpathroot} for the three diagrams under consideration.\footnote{For $S \la X \ra Y$ and $S \ra X \ra Y$, Theorem \ref{the:condpathroot} gives $\beta_{YX \cdot S} = \sigma_{XY \cdot S} / \sigma^2_{X \cdot S} = \sigma_{XY} / \sigma^2_X = \beta_{YX}$. For $X \ra Y \la S$, Theorem \ref{the:condpathroot} and $X \ci S | \emptyset$ give $\beta_{YX \cdot S} = \sigma_{XY \cdot S} / \sigma^2_{X \cdot S} = \sigma_{XY} / \sigma^2_{X} = \beta_{YX}$.} Moreover, \citet{Pearl2014} does not discuss if Simpson's paradox can occur for the diagram $X \ra Y \ra S$. In fact, we can use Theorem \ref{the:condpathroot} again to conclude that collapsibility does not hold for this diagram, regardless of whether association means covariance or regression coefficient.\footnote{For $X \ra Y \ra S$, Theorem \ref{the:condpathroot} and $Y \nci S | \emptyset$ give $\sigma_{XY \cdot S} = \sigma_{XY} ( \sigma^2_{X} / \sigma^2_{X} ) ( \sigma^2_{Y \cdot S} / \sigma^2_{Y} ) \neq \sigma_{XY}$. Moreover, $\beta_{YX \cdot S} = \sigma_{XY \cdot S} / \sigma^2_{X \cdot S} = ( \sigma_{XY} / \sigma^2_{X \cdot S} )( \sigma^2_{Y \cdot S} / \sigma^2_{Y} ) \neq \sigma_{XY} / \sigma^2_X = \beta_{YX}$, because $X \nci S | \emptyset$ and $Y \nci S | \emptyset$.} Pearl does not discuss either the case of singly-connected path diagrams where $X$ and $Y$ are connected by a path of length greater than one. We fill these gaps next.

Note that Simpson's paradox concerns the sign of the regression coefficient of $Y$ on $X$ upon conditioning on $S$ or, equivalently, it concerns the sign of the covariance between $X$ and $Y$ upon conditioning on $S$. Therefore, we are interested in the collapsibility of their sign rather than in the collapsibility of the regression coefficient or the covariance themselves. Corollary \ref{cor:samesign} implies that conditioning on $S$ does not change the sign of the covariance if $S$ satisfies the conditions in Theorem \ref{the:condpathroot} or \ref{the:condpathnonroot}. Consequently, Simpson's paradox cannot occur in those cases which, recall from Section \ref{sec:main}, include singly-connected path diagrams.

\section{Discussion}\label{sec:discussion}

In this work, we have extended path analysis by giving sufficient conditions for computing the partial covariance of two random variables from their covariance. This is done by correcting the covariance with the product of some partial variance ratios. These ratios can be interpreted as deflation factors that account for the reduction of the partial variances of the variables in the paths between the two variables of interest. As a result, the partial covariance retains the covariance's salient feature of factorizing over the edges in the paths. Moreover, we have used these results to show that Simpson's paradox cannot occur under the sufficient conditions developed.

\begin{figure}
\begin{tabular}{c|c}
\begin{tabular}{c}
\begin{tikzpicture}[inner sep=1mm]
\node at (0,0) (X) {$X$};
\node at (1.5,0) (S) {$R$};
\node at (3,0) (Y) {$Y$};
\node at (1.5,-1.5) (W) {$S$};
\path[<-] (W) edge (S);
\path[->] (S) edge (Y);
\path[->] (S) edge [bend right] (X);
\path[<->] (S) edge [bend left] (X);
\end{tikzpicture}
\end{tabular}
&
\begin{tabular}{c}
\begin{tikzpicture}[inner sep=1mm]
\node at (0,0) (X) {$X$};
\node at (1.5,0) (S) {$R$};
\node at (3,0) (Y) {$Y$};
\node at (1.5,-1.5) (W) {$S$};
\path[->] (X) edge (S);
\path[->] (S) edge (Y);
\path[->] (S) edge [bend right] (W);
\path[<->] (S) edge [bend left] (W);
\end{tikzpicture}
\end{tabular}
\end{tabular}\caption{Path diagrams conditioned on $S$ where Theorems \ref{the:condpathroot} and \ref{the:condpathnonroot} cannot be applied to compute $\sigma_{XY \cdot S}$.}\label{fig:discussion}
\end{figure}
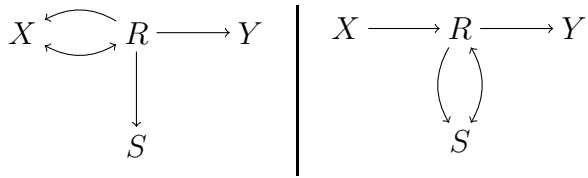

In the future, we would like to relax these sufficient conditions. Specifically, we would like to address the following two limitations of Theorems \ref{the:condpathroot} and \ref{the:condpathnonroot}. A limitation is that each node $X_i$ in the subpath shared by all the paths in $\Pi_{X:Y}$ must be either the root or a non-root in all the paths in $\Pi_{X:Y}$. As a consequence, the theorems do not apply to the path diagram conditioned on $S$ that is shown to the left in Figure \ref{fig:discussion}, because $R$ is the root of the path $X \la R \ra Y$ but it is a non-root node in the path $X \aa R \ra Y$. Ideally, we would like to apply Theorem \ref{the:condpathroot} to just the first path, and Theorem \ref{the:condpathnonroot} to just the second. This would imply a correction due to the first path and no correction due to the second (the partial variance ratio equals 1). Developing such a path-wise correction is an open question.

A related limitation of Theorems \ref{the:condpathroot} and \ref{the:condpathnonroot} is that there cannot be any $Z$-open route $X_i \ra A \oo \cdots \oo B \oa X_i$ with $i>1$. As a consequence, the theorems do not apply to the path diagram conditioned on $S$ that is shown to the right in Figure \ref{fig:discussion}, because $S$ is both a child and a spouse of $R$. Ideally, we would like to apply Theorem \ref{the:condpathroot} to account for $S$ as a spouse, and Theorem \ref{the:condpathnonroot} to account for $S$ as a child. This would imply no correction in the first case (the partial variance ratio equals 1) and a correction in the second case. Developing such a role-wise correction is an open question.

\section*{Appendix: Proofs}

This appendix contains the proofs of Theorems \ref{the:condpathroot} and \ref{the:condpathnonroot}. We start with some lemmas stating some auxiliary results. The proofs of the lemmas contain some repetition. We decided to keep it this way for the sake of clarity. Recall from Footnote \ref{foo:subpath} that when we say that every path has a subpath $A \oa$, we do not mean that the successor of $A$ is the same in every path. We mean that the successor is reached through an edge $\ra$ or $\aa$ in every path. Given a route $\rho_{X:Y}$ from $X$ to $Y$, we let $\rho_{X:A}$ denote the subroute of $\rho_{X:Y}$ from $X$ to $A$. Given two routes $\rho_{X:A}$ and $\rho_{A:Y}$, we let $\rho_{X:A} \cup \rho_{A:Y}$ denote the route from $X$ to $Y$ resulting from concatenating $\rho_{X:A}$ and $\rho_{A:Y}$. Finally, the path corresponding to a $Z$-open route from $X$ to $Y$ is a $Z$-open path from $X$ to $Y$ whose edges are a subset of the edges in the route. Such a path always exists by Lemma \ref{lem:pathroute}.

\begin{lemma}\label{lem:pathroute}
There is a $Z$-open route from $X$ to $Y$ if and only if there is a $Z$-open path from $X$ to $Y$. Moreover, the path and the route can be chosen such that the edges in the former are a subset of the edges in the latter.
\end{lemma}

\begin{proof}
Let $\pi_{X:Y}$ be a $Z$-open path from $X$ to $Y$. For every subpath $A \oa C \ao B$ of $\pi_{X:Y}$ such that $C \notin Z$, do the following. First, find a path $C \ra X_1 \ra \cdots \ra X_n$ with $X_n \in Z$ and $X_1, \ldots, X_{n-1} \notin Z$. Such a path must exist for $\pi_{X:Y}$ to be $Z$-open. Second, replace $A \oa C \ao B$ with $A \oa C \ra X_1 \ra \cdots \ra X_n \la \cdots \la X_1 \la C \ao B$. The result is the desired route.

Let $\rho_{X:Y}$ be a $Z$-open route from $X$ to $Y$. Repeat the following two steps while possible. The result is the desired path. First, choose a node $A$ that occurs several times in $\rho_{X:Y}$. Let $A_1$ and $A_2$ denote the first and the last occurrences of $A$ in $\rho_{X:Y}$. Assume without loss of generality that no node in $\rho_{X:A_1}$ occurs in $\rho_{A_2:Y}$ except $A$. Second, consider the following cases.
\begin{itemize}
\item If $\rho_{X:Y}$ is $X \oo \cdots \la A_1 \oo \cdots \oo A_2 \oo \cdots \oo Y$, then replace it with $\rho_{X:A_1} \cup \rho_{A_2:Y}$.

\item If $\rho_{X:Y}$ is $X \oo \cdots \oa A_1 \oo \cdots \oo A_2 \ra \cdots \oo Y$, then replace it with $\rho_{X:A_1} \cup \rho_{A_2:Y}$.

\item If $\rho_{X:Y}$ is $X \oo \cdots \oa A_1 \oo \cdots \oo A_2 \ao \cdots \oo Y$, then replace it with $\rho_{X:A_1} \cup \rho_{A_2:Y}$. Note that $A$ or some descendant of $A$ must be in $Z$ for the original $\rho_{X:Y}$ to be $Z$-open.
\end{itemize}
\end{proof}

\begin{lemma}\label{lem:aux1}
Consider a path diagram. Let $X$, $Y$, $R$ and $W$ be nodes and $Z$ a set of nodes. If $X \ci W | Z \cup R$ and $Y \ci W | Z \cup R$ and $X \ci Y | Z \cup R$, then
\[
\sigma_{XY \cdot ZW} =  \sigma_{XY \cdot Z} \frac{\sigma^2_{R \cdot ZW}}{\sigma^2_{R \cdot Z}}.
\]
\end{lemma}

\begin{proof}
Note that $X \ci W | Z \cup R$ implies that
\[
0 = \sigma_{XW \cdot ZR} = \sigma_{XW \cdot Z} - \frac{\sigma_{XR \cdot Z} \sigma_{RW \cdot Z}}{\sigma^2_{R \cdot Z}}
\]
which implies that $\sigma_{XW \cdot Z} = \delta_{XR \cdot Z} \sigma_{RW \cdot Z}$ where $\delta_{XR \cdot Z} = \sigma_{XR \cdot Z} / \sigma^2_{R \cdot Z}$. Likewise, $Y \ci W | Z \cup R$ implies that $\sigma_{YW \cdot Z} = \delta_{YR \cdot Z} \sigma_{RW \cdot Z}$ where $\delta_{YR \cdot Z} = \sigma_{YR \cdot Z} / \sigma^2_{R \cdot Z}$. Likewise, $X \ci Y | Z \cup R$ implies that $\sigma_{XY \cdot Z} = \delta_{XR \cdot Z} \delta_{YR \cdot Z} \sigma^2_{R \cdot Z}$. Therefore,
\begin{align*}
\sigma_{XY \cdot ZW} &= \sigma_{XY \cdot Z} - \frac{\sigma_{XW \cdot Z} \sigma_{WY \cdot Z}}{\sigma^2_{W \cdot Z}}\\
& = \delta_{XR \cdot Z} \delta_{YR \cdot Z} \sigma^2_{R \cdot Z} - \frac{\delta_{XR \cdot Z} \sigma_{RW \cdot Z} \delta_{YR \cdot Z} \sigma_{RW \cdot Z}}{\sigma^2_{W \cdot Z}}\\
& = \delta_{XR \cdot Z} \delta_{YR \cdot Z} \Big( \sigma^2_{R \cdot Z} - \frac{\sigma_{RW \cdot Z} \sigma_{RW \cdot Z}}{\sigma^2_{W \cdot Z}} \Big)\\
& = \delta_{XR \cdot Z} \delta_{YR \cdot Z} \sigma^2_{R \cdot ZW} =  \sigma_{XY \cdot Z} \frac{\sigma^2_{R \cdot ZW}}{\sigma^2_{R \cdot Z}}.
\end{align*}
\end{proof}

\begin{lemma}\label{lem:aux2}
Consider a path diagram. Let $X$, $Y$ and $W$ be nodes and $Z$ a set of nodes. If $Y \ci W | Z \cup X$, then
\[
\sigma_{XY \cdot ZW} =  \sigma_{XY \cdot Z} \frac{\sigma^2_{X \cdot ZW}}{\sigma^2_{X \cdot Z}}.
\]
\end{lemma}

\begin{proof}
Note that $Y \ci W | Z \cup X$ implies that
\[
0 = \sigma_{YW \cdot ZX} = \sigma_{YW \cdot Z} - \frac{\sigma_{YX \cdot Z} \sigma_{XW \cdot Z}}{\sigma^2_{X \cdot Z}}
\]
which implies that
\[
\sigma_{YW \cdot Z} = \frac{\sigma_{YX \cdot Z} \sigma_{XW \cdot Z}}{\sigma^2_{X \cdot Z}}.
\]
Therefore,
\begin{align*}
\sigma_{XY \cdot ZW} &= \sigma_{XY \cdot Z} - \frac{\sigma_{XW \cdot Z} \sigma_{WY \cdot Z}}{\sigma^2_{W \cdot Z}} = \sigma_{XY \cdot Z} - \frac{\sigma_{XW \cdot Z} \sigma_{YX \cdot Z} \sigma_{XW \cdot Z}}{\sigma^2_{W \cdot Z} \sigma^2_{X \cdot Z}}\\
&= \sigma_{XY \cdot Z} \Big( 1 - \frac{\sigma_{XW \cdot Z} \sigma_{XW \cdot Z}}{\sigma^2_{W \cdot Z} \sigma^2_{X \cdot Z}} \Big) = \frac{\sigma_{XY \cdot Z}}{\sigma^2_{X \cdot Z}} \Big( \sigma^2_{X \cdot Z} - \frac{\sigma_{XW \cdot Z} \sigma_{XW \cdot Z}}{\sigma^2_{W \cdot Z}} \Big)\\
&= \sigma_{XY \cdot Z} \frac{\sigma^2_{X \cdot ZW}}{\sigma^2_{X \cdot Z}}.
\end{align*}
\end{proof}

\begin{lemma}\label{lem:aux3}
Consider a path diagram. Let $X$, $Y$ and $W$ be nodes and $Z$ a set of nodes. If $X \ci W | Z$ or $Y \ci W | Z$, then $\sigma_{XY \cdot ZW} =  \sigma_{XY \cdot Z}$.
\end{lemma}

\begin{proof}
Note that $X \ci W | Z$ and $Y \ci W | Z$ imply $\sigma_{XW \cdot Z}=0$ and $\sigma_{YW \cdot Z}=0$, respectively. Thus, if $X \ci W | Z$ or $Y \ci W | Z$ then
\[
\sigma_{XY \cdot ZW} = \sigma_{XY \cdot Z} - \frac{\sigma_{XW \cdot Z} \sigma_{WY \cdot Z}}{\sigma^2_{W \cdot Z}} = \sigma_{XY \cdot Z}.
\]
\end{proof}

\begin{lemma}\label{lem:root1}
Consider a path diagram. Let $\Pi_{X:Y}$ denote all the $Z$-open paths from $X$ to $Y$. Suppose that no path in $\Pi_{X:Y}$ has colliders. Suppose that all the paths in $\Pi_{X:Y}$ have a subpath $\la R \ra$ or $R = X \ra$. Let $W$ be a node that is $Z$-connected to $R$ by a path that does not contain any node that is in some path in $\Pi_{X:Y}$ except $R$. If $Pa(W) \cup Sp(W)= \emptyset$, then
\[
\sigma_{XY \cdot ZW} =  \sigma_{XY \cdot Z} \frac{\sigma^2_{R \cdot ZW}}{\sigma^2_{R \cdot Z}}.
\]
\end{lemma}

\begin{proof}
Consider first the case where $\la R \ra$ is a subpath of every path in $\Pi_{X:Y}$. Assume to the contrary that $X \nci W | Z \cup R$ and let $\rho_{X:W}$ be a $(Z \cup R)$-open {\bf route}. Follow $\rho_{X:W}$ until reaching $R$ or $W$, and let $\pi_{X:Y} \in \Pi_{X:Y}$.
\begin{itemize}
\item If $R$ is reached first, then consider the first occurrence of $R$ in $\rho_{X:W}$ and note that $\rho_{X:R} \cup \pi_{R:Y}$ is a $Z$-open route. Moreover, it does not contain any edge $\la R$ that is in some path in $\Pi_{X:Y}$ because, otherwise, $\rho_{X:W}$ contains the edge ($\pi_{R:Y}$ cannot by definition) and, thus, it is not $(Z \cup R)$-open. Then, the path corresponding to $\rho_{X:R} \cup \pi_{R:Y}$ contradicts the assumptions in the lemma.

\item If $W$ is reached first, then let $\varrho_{R:W}$ denote a $Z$-open path that does not contain any node that is in some path in $\Pi_{X:Y}$ except $R$. Such a path exists by the assumptions in the lemma. Thus, $\rho_{X:W} \cup \varrho_{W:R} \cup \pi_{R:Y}$ is a $Z$-open route, because neither $R$ nor $W$ is a collider in it. The latter follows from the assumption that $Pa(W) \cup Sp(W)=\emptyset$. Moreover, the route does not contain any edge $\la R$ that is in some path in $\Pi_{X:Y}$ because, otherwise, $\rho_{X:W}$ contains the edge ($\varrho_{W:R}$ and $\pi_{R:Y}$ cannot by definition) and thus it reaches $R$ first. Then, the path corresponding to $\rho_{X:W} \cup \varrho_{W:R} \cup \pi_{R:Y}$ contradicts the assumptions in the lemma.
\end{itemize}
Consequently, $X \ci W | Z \cup R$. We can analogously prove that $Y \ci W | Z \cup R$. Now, assume to the contrary that $X \nci Y | Z \cup R$ and let $\rho_{X:Y}$ be a $(Z \cup R)$-open {\bf route}. Note that $R$ must be in $\rho_{X:Y}$ because, otherwise, $\rho_{X:Y}$ is $Z$-open and, thus, its corresponding path contradicts the assumptions in the lemma. Then, consider the first occurrence of $R$ in $\rho_{X:Y}$ and note that $\rho_{X:R} \cup \pi_{R:Y}$ is a $Z$-open route. Moreover, the route does not contain any edge $\la R$ that is in some path in $\Pi_{X:Y}$ because, otherwise, $\rho_{X:Y}$ contains the edge ($\pi_{R:Y}$ cannot by definition) and thus it is not $(Z \cup R)$-open. Then, the path corresponding to $\rho_{X:R} \cup \pi_{R:Y}$ contradicts the assumptions in the lemma. Consequently, $X \ci Y | Z \cup R$. Therefore, the desired result follows from Lemma \ref{lem:aux1}.

Finally, consider the case where $R = X \ra$ is a subpath of every path in $\Pi_{X:Y}$. Assume to the contrary that $Y \nci W | Z \cup X$ and let $\rho_{Y:W}$ be a $(Z \cup X)$-open {\bf route}. Follow $\rho_{Y:W}$ until reaching $X$ or $W$.
\begin{itemize}
\item If $X$ is reached first, then note that $\rho_{Y:X}$ does not contain any edge $X \ra$ that is in some path in $\Pi_{X:Y}$ because, otherwise, $\rho_{Y:W}$ is not $(Z \cup X)$-open. Then, the path corresponding to $\rho_{X:Y}$ contradicts the assumptions in the lemma.

\item If $W$ is reached first, let $\varrho_{X:W}$ denote a $Z$-open path that does not contain any node that is in some path in $\Pi_{X:Y}$ except $X$. Such a path exists by the assumptions in the lemma. Then, $\varrho_{X:W} \cup \rho_{W:Y}$ is a $Z$-open route, because $W$ is not a collider in it due to the assumption that $Pa(W) \cup Sp(W)=\emptyset$. Moreover, the route does not contain any edge $X \ra$ that is in some path in $\Pi_{X:Y}$ because, otherwise, $\rho_{Y:W}$ contains the edge ($\varrho_{X:W}$ cannot by definition) and thus it reaches $X$ first. Then, the path corresponding to $\varrho_{X:W} \cup \rho_{W:Y}$ contradicts the assumptions in the lemma.
\end{itemize}
Consequently, $Y \ci W | Z \cup X$ and, thus, the desired result follows from Lemma \ref{lem:aux2}.
\end{proof}

\begin{lemma}\label{lem:root2}
Consider a path diagram. Let $\Pi_{X:Y}$ denote all the $Z$-open paths from $X$ to $Y$. Suppose that no path in $\Pi_{X:Y}$ has colliders. Suppose that all the paths in $\Pi_{X:Y}$ have a subpath $\la R \ra$ or $R = X \ra$. Let $W$ be a node that is $Z$-connected to $R$ by a path that does not contain any node that is in some path in $\Pi_{X:Y}$ except $R$. If $Ch(W)= \emptyset$ and $\Pi_{X:Y}$ are all the $(Z \cup W)$-open paths from $X$ to $Y$, then
\[
\sigma_{XY \cdot ZW} =  \sigma_{XY \cdot Z} \frac{\sigma^2_{R \cdot ZW}}{\sigma^2_{R \cdot Z}}.
\]
\end{lemma}

\begin{proof}
Consider first the case where $\la R \ra$ is a subpath of every path in $\Pi_{X:Y}$. Assume to the contrary that $X \nci W | Z \cup R$ and let $\rho_{X:W}$ be a $(Z \cup R)$-open {\bf route}. Follow $\rho_{X:W}$ until reaching $R$ or $W$, and let $\pi_{X:Y} \in \Pi_{X:Y}$.
\begin{itemize}
\item If $R$ is reached first, then consider the first occurrence of $R$ in $\rho_{X:W}$ and note that $\rho_{X:R} \cup \pi_{R:Y}$ is a $Z$-open route. Moreover, it does not contain any edge $\la R$ that is in some path in $\Pi_{X:Y}$ because, otherwise, $\rho_{X:W}$ contains the edge ($\pi_{R:Y}$ cannot by definition) and, thus, it is not $(Z \cup R)$-open. Then, the path corresponding to $\rho_{X:R} \cup \pi_{R:Y}$ contradicts the assumptions in the lemma.

\item If $W$ is reached first, then let $\varrho_{R:W}$ denote a $Z$-open path that does not contain any node that is in some path in $\Pi_{X:Y}$ except $R$. Such a path exists by the assumptions in the lemma. Thus, $\rho_{X:W} \cup \varrho_{W:R} \cup \pi_{R:Y}$ is a $(Z \cup W)$-open route from $X$ to $Y$, because $W$ is a collider in it whereas $R$ is not. The former follows from the assumption that $Ch(W)=\emptyset$. Moreover, the route does not contain any edge $\la R$ that is in some path in $\Pi_{X:Y}$ because, otherwise, $\rho_{X:W}$ contains the edge ($\varrho_{W:R}$ and $\pi_{R:Y}$ cannot by definition) and thus it reaches $R$ first. Then, the path corresponding to $\rho_{X:W} \cup \varrho_{W:R} \cup \pi_{R:Y}$ contradicts the assumptions in the lemma.
\end{itemize}
Consequently, $X \ci W | Z \cup R$. We can analogously prove that $Y \ci W | Z \cup R$. Now, assume to the contrary that $X \nci Y | Z \cup R$ and let $\rho_{X:Y}$ be a $(Z \cup R)$-open {\bf route}. Note that $R$ must be in $\rho_{X:Y}$ because, otherwise, $\rho_{X:Y}$ is $Z$-open and, thus, its corresponding path contradicts the assumptions in the lemma. Then, consider the first occurrence of $R$ in $\rho_{X:Y}$ and note that $\rho_{X:R} \cup \pi_{R:Y}$ is a $Z$-open route. Moreover, the route does not contain any edge $\la R$ that is in some path in $\Pi_{X:Y}$ because, otherwise, $\rho_{X:Y}$ contains the edge ($\pi_{R:Y}$ cannot by definition) and thus it is not $(Z \cup R)$-open. Then, the path corresponding to $\rho_{X:R} \cup \pi_{R:Y}$ contradicts the assumptions in the lemma. Consequently, $X \ci Y | Z \cup R$. Therefore, the desired result follows from Lemma \ref{lem:aux1}.

Finally, consider the case where $R = X \ra$ is a subpath of every path in $\Pi_{X:Y}$. Assume to the contrary that $Y \nci W | Z \cup X$ and let $\rho_{Y:W}$ be a $(Z \cup X)$-open {\bf route}. Follow $\rho_{Y:W}$ until reaching $X$ or $W$.
\begin{itemize}
\item If $X$ is reached first, then note that $\rho_{Y:X}$ does not contain any edge $X \ra$ that is in some path in $\Pi_{X:Y}$ because, otherwise, $\rho_{Y:W}$ is not $(Z \cup X)$-open. Then, the path corresponding to $\rho_{X:Y}$ contradicts the assumptions in the lemma.

\item If $W$ is reached first, let $\varrho_{X:W}$ denote a $Z$-open path that does not contain any node that is in some path in $\Pi_{X:Y}$ except $X$. Such a path exists by the assumptions in the lemma. Then, $\varrho_{X:W} \cup \rho_{W:Y}$ is a $(Z \cup W)$-open route, because $W$ is a collider in it due to the assumption that $Ch(W)=\emptyset$. Moreover, the route does not contain any edge $X \ra$ that is in some path in $\Pi_{X:Y}$ because, otherwise, $\rho_{Y:W}$ contains the edge ($\varrho_{X:W}$ cannot by definition) and thus it reaches $X$ first. Then, the path corresponding to $\varrho_{X:W} \cup \rho_{W:Y}$ contradicts the assumptions in the lemma.
\end{itemize}
Consequently, $Y \ci W | Z \cup X$ and, thus, the desired result follows from Lemma \ref{lem:aux2}.
\end{proof}

\begin{lemma}\label{lem:nonroot1}
Consider a path diagram. Let $\Pi_{X:Y}$ denote all the $Z$-open paths from $X$ to $Y$. Suppose that no path in $\Pi_{X:Y}$ has colliders. Suppose that all the paths in $\Pi_{X:Y}$ have a subpath $\oa R \ra$ or $\oa Y = R$. Let $W$ be a node that is $Z$-connected to $R$ through $Pa(R) \cup Sp(R)$ by a path that does not contain any node that is in some path in $\Pi_{X:Y}$ except $R$. If $Pa(W) \cup Sp(W)= \emptyset$ and $W$ is not $Z$-connected to $R$ through $Ch(R)$, then
\[
\sigma_{XY \cdot ZW} =  \sigma_{XY \cdot Z}.
\]
\end{lemma}

\begin{proof}
Consider first the case where $\oa R \ra$ is a subpath of every path in $\Pi_{X:Y}$. Assume to the contrary that $X \nci W | Z$ and let $\rho_{X:W}$ be a $Z$-open {\bf route}. Follow $\rho_{X:W}$ until reaching $R$ or $W$, and let $\pi_{X:Y} \in \Pi_{X:Y}$.
\begin{itemize}
\item If $R$ is reached first, then consider the first occurrence of $R$ in $\rho_{X:W}$ and note that $\rho_{X:R}$ ends with an edge $\oa R$ because, otherwise, $\rho_{X:R} \cup \pi_{R:Y}$ is a $Z$-open route that has a subroute $\la R \ra$ and, thus, its corresponding path contradicts the assumptions in the lemma. Then, $\rho_{R:W}$ must start with an edge $R \ra$ for $\rho_{X:W}$ to be $Z$-open. However, this contradicts the assumption that $W$ is not $Z$-connected to $R$ through $Ch(R)$.

\item If $W$ is reached first, then let $\varrho_{R:W}$ denote a $Z$-open path that does not contain any node that is in some path in $\Pi_{X:Y}$ except $R$. Such a path exists by the assumptions in the lemma. Thus, $\rho_{X:W} \cup \varrho_{W:R} \cup \pi_{R:Y}$ is a $Z$-open route, because neither $R$ nor $W$ is a collider in it. The latter follows from the assumption that $Pa(W) \cup Sp(W)=\emptyset$. Moreover, the route does not contain any edge $\oa R$ that is in some path in $\Pi_{X:Y}$ because, otherwise, $\rho_{X:W}$ contains the edge ($\varrho_{W:R}$ and $\pi_{R:Y}$ cannot by definition) and thus it reaches $R$ first. Then, the path corresponding to $\rho_{X:W} \cup \varrho_{W:R} \cup \pi_{R:Y}$ contradicts the assumptions in the lemma.
\end{itemize}
Consequently, $X \ci W | Z$. When $\oa Y = R$ is a subpath of every path in $\Pi_{X:Y}$, we can prove that $X \ci W | Z$ much in the same way. Consequently, $X \ci W | Z$ in either case and, thus, the desired result follows from Lemma \ref{lem:aux3}.
\end{proof}

\begin{lemma}\label{lem:nonroot2}
Consider a path diagram. Let $\Pi_{X:Y}$ denote all the $Z$-open paths from $X$ to $Y$. Suppose that no path in $\Pi_{X:Y}$ has colliders. Suppose that all the paths in $\Pi_{X:Y}$ have a subpath $\oa R \ra$ or $\oa Y = R$. Let $W$ be a node that is $Z$-connected to $R$ through $Pa(R) \cup Sp(R)$ by a path that does not contain any node that is in some path in $\Pi_{X:Y}$ except $R$. If $Ch(W)= \emptyset$, and $W$ is not $Z$-connected to $R$ through $Ch(R)$, and $\Pi_{X:Y}$ are all the $(Z \cup W)$-open paths from $X$ to $Y$, then
\[
\sigma_{XY \cdot ZW} =  \sigma_{XY \cdot Z}.
\]
\end{lemma}

\begin{proof}
Consider first the case where $\oa R \ra$ is a subpath of every path in $\Pi_{X:Y}$. Assume to the contrary that $X \nci W | Z$ and let $\rho_{X:W}$ be a $Z$-open {\bf route}. Follow $\rho_{X:W}$ until reaching $R$ or $W$, and let $\pi_{X:Y} \in \Pi_{X:Y}$.
\begin{itemize}
\item If $R$ is reached first, then consider the first occurrence of $R$ in $\rho_{X:W}$ and note that $\rho_{X:R}$ ends with an edge $\oa R$ because, otherwise, $\rho_{X:R} \cup \pi_{R:Y}$ is a $Z$-open route that has a subroute $\la R \ra$ and, thus, its corresponding path contradicts the assumptions in the lemma. Then, $\rho_{R:W}$ must start with an edge $R \ra$ for $\rho_{X:W}$ to be $Z$-open. However, this contradicts the assumption that $W$ is not $Z$-connected to $R$ through $Ch(R)$.

\item If $W$ is reached first, then let $\varrho_{R:W}$ denote a $Z$-open path that does not contain any node that is in some path in $\Pi_{X:Y}$ except $R$. Such a path exists by the assumptions in the lemma. Thus, $\rho_{X:W} \cup \varrho_{W:R} \cup \pi_{R:Y}$ is a $(Z \cup W)$-open route, because $W$ is a collider in it whereas $R$ is not. The former follows from the assumption that $Ch(W)=\emptyset$. Moreover, the route does not contain any edge $\oa R$ that is in some path in $\Pi_{X:Y}$ because, otherwise, $\rho_{X:W}$ contains the edge ($\varrho_{W:R}$ and $\pi_{R:Y}$ cannot by definition) and thus it reaches $R$ first. Then, the path corresponding to $\rho_{X:W} \cup \varrho_{W:R} \cup \pi_{R:Y}$ contradicts the assumptions in the lemma.
\end{itemize}
Consequently, $X \ci W | Z$. When $\oa Y = R$ is a subpath of every path in $\Pi_{X:Y}$, we can prove that $X \ci W | Z$ much in the same way. Consequently, $X \ci W | Z$ in either case and, thus, the desired result follows from Lemma \ref{lem:aux3}.
\end{proof}

\begin{lemma}\label{lem:nonroot3}
Consider a path diagram. Let $\Pi_{X:Y}$ denote all the $Z$-open paths from $X$ to $Y$. Suppose that no path in $\Pi_{X:Y}$ has colliders. Suppose that all the paths in $\Pi_{X:Y}$ have a subpath $\oa R \ra$ or $\oa Y = R$. Let $W$ be a node that is $Z$-connected to $R$ through $Ch(R)$ by a path that does not contain any node that is in some path in $\Pi_{X:Y}$ except $R$. If $Pa(W) \cup Sp(W)= \emptyset$ and $W$ is not $Z$-connected to $R$ through $Pa(R) \cup Sp(R)$, then
\[
\sigma_{XY \cdot ZW} =  \sigma_{XY \cdot Z} \frac{\sigma^2_{R \cdot ZW}}{\sigma^2_{R \cdot Z}}.
\]
\end{lemma}

\begin{proof}
Consider first the case where $\oa R \ra$ is a subpath of every path in $\Pi_{X:Y}$. Assume to the contrary that $X \nci W | Z \cup R$ and let $\rho_{X:W}$ be a $(Z \cup R)$-open {\bf route}. Follow $\rho_{X:W}$ until reaching $R$ or $W$, and let $\pi_{X:Y} \in \Pi_{X:Y}$.
\begin{itemize}
\item If $R$ is reached first, then note $R$ must be a collider in $\rho_{X:W}$ for this to be $(Z \cup R)$-open. However, the last occurrence of $R$ in $\rho_{X:W}$ contradicts the assumption that $W$ is not $Z$-connected to $R$ through $Pa(R) \cup Sp(R)$.

\item If $W$ is reached first, then let $\varrho_{R:W}$ denote a $Z$-open path that does not contain any node that is in some path in $\Pi_{X:Y}$ except $R$. Such a path exists by the assumptions in the lemma. Thus, $\rho_{X:W} \cup \varrho_{W:R} \cup \pi_{R:Y}$ is a $Z$-open route, because neither $R$ nor $W$ is a collider in it. The latter follows from the assumption that $Pa(W) \cup Sp(W)=\emptyset$. Moreover, the route does not contain any edge $\oa R$ that is in some path in $\Pi_{X:Y}$ because, otherwise, $\rho_{X:W}$ contains the edge ($\varrho_{W:R}$ and $\pi_{R:Y}$ cannot by definition) and thus it reaches $R$ first. Then, the path corresponding to $\rho_{X:W} \cup \varrho_{W:R} \cup \pi_{R:Y}$ contradicts the assumptions in the lemma.
\end{itemize}
Consequently, $X \ci W | Z \cup R$. Now, assume to the contrary that $Y \nci W | Z \cup R$ and let $\rho_{Y:W}$ be a $(Z \cup R)$-open {\bf route}. Follow $\rho_{Y:W}$ until reaching $R$ or $W$, and let $\pi_{X:Y} \in \Pi_{X:Y}$.
\begin{itemize}
\item If $R$ is reached first, then note $R$ must be a collider in $\rho_{Y:W}$ for this be $(Z \cup R)$-open. However, the last occurrence of $R$ in $\rho_{Y:W}$ contradicts the assumption that $W$ is not $Z$-connected to $R$ through $Pa(R) \cup Sp(R)$.

\item If $W$ is reached first, then let $\varrho_{R:W}$ denote a $Z$-open path that leaves $R$ through $Ch(R)$ and that does not contain any node that is in some path in $\Pi_{X:Y}$ except $R$. Such a path exists by the assumptions in the lemma. Thus, $\pi_{X:R} \cup \varrho_{R:W} \cup \rho_{W:Y}$ is a $Z$-open route, because neither $R$ nor $W$ is a collider in it. The latter follows from the assumption that $Pa(W) \cup Sp(W)=\emptyset$. Moreover, the route does not contain any edge $R \ra$ that is in some path in $\Pi_{X:Y}$ because, otherwise, $\rho_{W:Y}$ contains the edge ($\pi_{X:R}$ and $\varrho_{R:W}$ cannot by definition) and thus it reaches $R$ first. Then, the path corresponding to $\pi_{X:R} \cup \varrho_{R:W} \cup \rho_{W:Y}$ contradicts the assumptions in the lemma.
\end{itemize}
Consequently, $Y \ci W | Z \cup R$. Now, assume to the contrary that $X \nci Y | Z \cup R$ and let $\rho_{X:Y}$ be a $(Z \cup R)$-open {\bf path}. Note that $R$ must be a collider or a descendant of a collider in $\rho_{X:Y}$ because, otherwise, $R$ is not in $\rho_{X:Y}$ and, thus, $\rho_{X:Y}$ is $Z$-open, which contradicts the assumptions in the lemma. Note also that the assumption that $W$ is $Z$-connected to $R$ through $Ch(R)$ implies that some descendant of $R$ is in $Z \cup W$. Actually, some descendant of $R$ must be in $Z$ due to the assumption that $Pa(W) \cup Sp(W)=\emptyset$. Then, $\rho_{X:Y}$ is $Z$-open, which contradicts the assumptions in the lemma. Consequently, $X \ci Y | Z \cup R$. Therefore, the desired result follows from Lemma \ref{lem:aux1}.

Finally, consider the case where $\oa Y = R$ is a subpath of every path in $\Pi_{X:Y}$. Assume to the contrary that $X \nci W | Z \cup Y$ and let $\rho_{X:W}$ be a $(Z \cup Y)$-open {\bf route}. Follow $\rho_{X:W}$ until reaching $Y$ or $W$.
\begin{itemize}
\item If $Y$ is reached first, then note $Y$ must be a collider in $\rho_{X:W}$ for this to be $(Z \cup Y)$-open. However, the last occurrence of $Y$ in $\rho_{X:W}$ contradicts the assumption that $W$ is not $Z$-connected to $R$ through $Pa(Y) \cup Sp(Y)$.

\item If $W$ is reached first, let $\varrho_{Y:W}$ denote a $Z$-open path that does not contain any node that is in some path in $\Pi_{X:Y}$ except $Y$. Such a path exists by the assumptions in the lemma. Then, $\varrho_{X:W} \cup \rho_{W:Y}$ is a $Z$-open route, because $W$ is not a collider in it due to the assumption that $Pa(W) \cup Sp(W)=\emptyset$. Moreover, the route does not contain any edge $\oa Y$ that is in some path in $\Pi_{X:Y}$ because, otherwise, $\rho_{X:W}$ contains the edge ($\varrho_{Y:W}$ cannot by definition) and thus it reaches $Y$ first. Then, the path corresponding to $\varrho_{X:W} \cup \rho_{W:Y}$ contradicts the assumptions in the lemma.
\end{itemize}
Consequently, $X \ci W | Z \cup Y$ and, thus, the desired result follows from Lemma \ref{lem:aux2}.
\end{proof}

\begin{lemma}\label{lem:nonroot4}
Consider a path diagram. Let $\Pi_{X:Y}$ denote all the $Z$-open paths from $X$ to $Y$. Suppose that no path in $\Pi_{X:Y}$ has colliders. Suppose that all the paths in $\Pi_{X:Y}$ have a subpath $\oa R \ra$ or $\oa Y = R$. Let $W$ be a node that is $Z$-connected to $R$ through $Ch(R)$ by a path that does not contain any node that is in some path in $\Pi_{X:Y}$ except $R$. If $Ch(W)= \emptyset$, and $W$ is not $Z$-connected to $R$ by any path that reaches $R$ through $Pa(R) \cup Sp(R)$, and $\Pi_{X:Y}$ are all the $(Z \cup W)$-open paths from $X$ to $Y$, then
\[
\sigma_{XY \cdot ZW} =  \sigma_{XY \cdot Z} \frac{\sigma^2_{R \cdot ZW}}{\sigma^2_{R \cdot Z}}.
\]
\end{lemma}

\begin{proof}
Consider first the case where $\oa R \ra$ is a subpath of every path in $\Pi_{X:Y}$. Assume to the contrary that $X \nci W | Z \cup R$ and let $\rho_{X:W}$ be a $(Z \cup R)$-open {\bf route}. Follow $\rho_{X:W}$ until reaching $R$ or $W$, and let $\pi_{X:Y} \in \Pi_{X:Y}$.
\begin{itemize}
\item If $R$ is reached first, then note that $R$ must be a collider in $\rho_{X:W}$ for this to be $(Z \cup R)$-open. However, the last occurrence of $R$ in $\rho_{X:W}$ contradicts the assumption that $W$ is not $Z$-connected to $R$ through $Pa(R) \cup Sp(R)$.

\item If $W$ is reached first, then let $\varrho_{R:W}$ denote a $Z$-open path that does not contain any node that is in some path in $\Pi_{X:Y}$ except $R$. Such a path exists by the assumptions in the lemma. Thus, $\rho_{X:W} \cup \varrho_{W:R} \cup \pi_{R:Y}$ is a $(Z \cup W)$-open route, because $W$ is a collider in it whereas $R$ is not. The former follows from the assumption that $Ch(W)=\emptyset$. Moreover, the route does not contain any edge $\oa R$ that is in some path in $\Pi_{X:Y}$ because, otherwise, $\rho_{X:W}$ contains the edge ($\varrho_{W:R}$ and $\pi_{R:Y}$ cannot by definition) and thus it reaches $R$ first. Then, the path corresponding to $\rho_{X:W} \cup \varrho_{W:R} \cup \pi_{R:Y}$ contradicts the assumptions in the lemma.
\end{itemize}
Consequently, $X \ci W | Z \cup R$. Now, assume to the contrary that $Y \nci W | Z \cup R$ and let $\rho_{Y:W}$ be a $(Z \cup R)$-open {\bf route}. Follow $\rho_{Y:W}$ until reaching $R$ or $W$, and let $\pi_{X:Y} \in \Pi_{X:Y}$.
\begin{itemize}
\item If $R$ is reached first, then note that $R$ must be a collider in $\rho_{Y:W}$ for this to be $(Z \cup R)$-open. However, the last occurrence of $R$ in $\rho_{Y:W}$ contradicts the assumption that $W$ is not $Z$-connected to $R$ through $Pa(R) \cup Sp(R)$.

\item If $W$ is reached first, then let $\varrho_{R:W}$ denote a $Z$-open path that leaves $R$ through $Ch(R)$ and that does not contain any node that is in some path in $\Pi_{X:Y}$ except $R$. Such a path exists by the assumptions in the lemma. Thus, $\pi_{X:R} \cup \varrho_{R:W} \cup \rho_{W:Y}$ is a $(Z \cup W)$-open route, because $W$ is a collider in it whereas $R$ is not. The former follows from the assumption that $Ch(W)=\emptyset$. Moreover, the route does not contain any edge $R \ra$ that is in some path in $\Pi_{X:Y}$ because, otherwise, $\rho_{W:Y}$ contains the edge ($\pi_{X:R}$ and $\varrho_{R:W}$ cannot by definition) and thus it reaches $R$ first. Then, the path corresponding to $\pi_{X:R} \cup \varrho_{R:W} \cup \rho_{W:Y}$ contradicts the assumptions in the lemma.
\end{itemize}
Consequently, $Y \ci W | Z \cup R$. Now, assume to the contrary that $X \nci Y | Z \cup R$ and let $\rho_{X:Y}$ be a $(Z \cup R)$-open {\bf path}. Note that $R$ must be a collider or a descendant of a collider in $\rho_{X:Y}$ because, otherwise, $R$ is not in $\rho_{X:Y}$ and, thus, $\rho_{X:Y}$ is $Z$-open, which contradicts the assumptions in the lemma. Note also that the assumption that $W$ is $Z$-connected to $R$ through $Ch(R)$ implies that some descendant of $R$ is in $Z \cup W$. Then, $\rho_{X:Y}$ is $(Z \cup W)$-open, which contradicts the assumptions in the lemma. Consequently, $X \ci Y | Z \cup R$. Therefore, the desired result follows from Lemma \ref{lem:aux1}.

Finally, consider the case where $\oa Y = R$ is a subpath of every path in $\Pi_{X:Y}$. Assume to the contrary that $X \nci W | Z \cup Y$ and let $\rho_{X:W}$ be a $(Z \cup Y)$-open {\bf route}. Follow $\rho_{X:W}$ until reaching $Y$ or $W$.
\begin{itemize}
\item If $Y$ is reached first, then note $Y$ must be a collider in $\rho_{X:W}$ for this to be $(Z \cup Y)$-open. However, the last occurrence of $Y$ in $\rho_{X:W}$ contradicts the assumption that $W$ is not $Z$-connected to $R$ through $Pa(Y) \cup Sp(Y)$.

\item If $W$ is reached first, let $\varrho_{Y:W}$ denote a $Z$-open path that does not contain any node that is in some path in $\Pi_{X:Y}$ except $Y$. Such a path exists by the assumptions in the lemma. Then, $\varrho_{X:W} \cup \rho_{W:Y}$ is a $(Z \cup W)$-open route, because $W$ is a collider in it due to the assumption that $Ch(W)=\emptyset$. Moreover, the route does not contain any edge $\oa Y$ that is in some path in $\Pi_{X:Y}$ because, otherwise, $\rho_{X:W}$ contains the edge ($\varrho_{Y:W}$ cannot by definition) and thus it reaches $Y$ first. Then, the path corresponding to $\varrho_{X:W} \cup \rho_{W:Y}$ contradicts the assumptions in the lemma.
\end{itemize}
Consequently, $X \ci W | Z \cup Y$ and, thus, the desired result follows from Lemma \ref{lem:aux2}.
\end{proof}

\begin{proof}[Proof of Theorem \ref{the:condpathroot}]
Consider hereinafter the path diagram conditioned on $S$. We first compute $\sigma_{XY \cdot Z^1}$ from $\sigma_{XY}$ by adding the nodes in $Z^1$ to the conditioning set in the order $W_1, W_2, \ldots$. The assumption that $\Pi_{X:Y}$ are all the $Z$-open paths from $X$ to $Y$ implies that $\Pi_{X:Y}$ are all the $(W_{1:j-1})$-open paths from $X$ to $Y$ because, otherwise, any other path cannot be closed afterwards which contradicts the assumption. To see it, note that a node $W \in Z$ does not close any open path, since there is no subpath $\oa W \ra$ or $\la W \ra$ due to the conditioning operation. Then,
\[
\sigma_{XY \cdot W_1} =  \sigma_{XY} \frac{\sigma^2_{X_1 \cdot W_1}}{\sigma^2_{X_1}}
\]
by Lemma \ref{lem:root1} if $Pa(W_1) \cup Sp(W_1) = \emptyset$, or Lemma \ref{lem:root2} if $Ch(W_1)=\emptyset$. Likewise, 
\[
\sigma_{XY \cdot W_1 W_2} =  \sigma_{XY \cdot W_1} \frac{\sigma^2_{X_1 \cdot W_1 W_2}}{\sigma^2_{X_1 \cdot W_1}}
\]
by Lemma \ref{lem:root1} or \ref{lem:root2}. Combining the last two equations gives
\[
\sigma_{XY \cdot W_1 W_2} =  \sigma_{XY} \frac{\sigma^2_{X_1 \cdot W_1 W_2}}{\sigma^2_{X_1}}.
\]
Continuing with this process for the rest of the nodes in $Z^1$ gives
\begin{equation}\label{eq:Z1a}
\sigma_{XY \cdot Z^1} =  \sigma_{XY} \frac{\sigma^2_{X_1 \cdot Z^1}}{\sigma^2_{X_1}}.
\end{equation}

Now, we compute $\sigma_{XY \cdot Z^1_1}$ from $\sigma_{XY \cdot Z^1}$ by adding the nodes in $Z_1$ to the conditioning set in the order $W_1, W_2, \ldots$. Recall from above that $\Pi_{X:Y}$ are all the $(Z^1 \cup W_{1:j-1})$-open paths from $X$ to $Y$. Then,
\[
\sigma_{XY \cdot Z^1 Z_1} =  \sigma_{XY \cdot Z^1} \frac{\sigma^2_{X_1 \cdot Z^1 Z_1}}{\sigma^2_{X_1 \cdot Z^1}}
\]
by repeating the reasoning that led to Equation \ref{eq:Z1a}. Moreover, combining the last two equations yields
\begin{equation}\label{eq:Z11}
\sigma_{XY \cdot Z^1_1} =  \sigma_{XY} \frac{\sigma^2_{X_1 \cdot Z^1_1}}{\sigma^2_{X_1}}.
\end{equation}

Now, we compute $\sigma_{XY \cdot Z^1_1 Z^2}$ from $\sigma_{XY \cdot Z^1_1}$ by adding the nodes in $Z^2$ to the conditioning set in the order $W_1, W_2, \ldots$. The assumption that there is no $Z$-open route $X_2 \ra A \oo \cdots \oo B \oa X_2$ implies that there is no $(Z^1_1 \cup W_{1:j-1})$-open path from $W_j$ to $X_2$ through $Ch(X_2)$. To see it, assume the opposite. Then, there are $(Z^1_1 \cup W_{1:j-1})$-open paths from $W_j$ to $X_2$ through both $Ch(X_2)$ and $Pa(X_2) \cup Sp(X_2)$. This implies that there is a route $X_2 \ra A \oo \cdots \oo B \oa X_2$ that contains $W_j$, and the route is $(Z^1_1 \cup W_{1:j-1})$-open or $(Z^1_1 \cup W_{1:j})$-open. Then, there is a path $\ra A \oo \cdots \oo B \oa$ that is $(Z^1_1 \cup W_{1:j-1})$-open or $(Z^1_1 \cup W_{1:j})$-open by Lemma \ref{lem:pathroute}. However, this path cannot be closed afterwards which contradicts the assumption. To see it, note that a node $W \in Z$ does not close any open path, since there is no subpath $\oa W \ra$ or $\la W \ra$ due to the conditioning operation. Consequently, there is no $(Z^1_1 \cup W_{1:j-1})$-open path from $W_j$ to $X_2$ through $Ch(X_2)$. Recall also from above that $\Pi_{X:Y}$ are all the $(Z^1_1 \cup W_{1:j-1})$-open paths from $X$ to $Y$. Then,
\[
\sigma_{XY \cdot Z^1_1 W_1} =  \sigma_{XY \cdot Z^1_1}.
\]
by Lemma \ref{lem:nonroot1} if $Pa(W_1) \cup Sp(W_1) = \emptyset$, or Lemma \ref{lem:nonroot2} if $Ch(W_1)=\emptyset$. Likewise,
\[
\sigma_{XY \cdot Z^1_1 W_1 W_2} =  \sigma_{XY \cdot Z^1_1 W_1}.
\]
by Lemma \ref{lem:nonroot1} or \ref{lem:nonroot2}. Combining the last two equations gives
\[
\sigma_{XY \cdot W_1 W_2} =  \sigma_{XY \cdot Z^1_1}.
\]
Continuing with this process for the rest of the nodes in $Z^2$ gives
\begin{equation}\label{eq:Z2a}
\sigma_{XY \cdot Z^1_1 Z^2} =  \sigma_{XY \cdot  Z^1_1}.
\end{equation}

Now, we compute $\sigma_{XY \cdot Z^1_1 Z^2_2}$ from $\sigma_{XY \cdot Z^1_1 Z^2}$ by adding the nodes in $Z_2$ to the conditioning set in the order $W_1, W_2, \ldots$. Recall from above that $\Pi_{X:Y}$ are all the $(Z^{1:2}_1 \cup W_{1:j-1})$-open paths from $X$ to $Y$. Recall also from above that the assumption that there is no $Z$-open route $X_2 \ra A \oo \cdots \oo B \oa X_2$ implies that there is no $(Z^{1:2}_1 \cup W_{1:j-1})$-open path from $W_j$ to $X_2$ through $Pa(X_2) \cup Sp(X_2)$. Then,
\[
\sigma_{XY \cdot Z^1_1 Z^2_2} =  \sigma_{XY \cdot Z^1_1 Z^2} \frac{\sigma^2_{X_2 \cdot Z^1_1 Z^2_2}}{\sigma^2_{X_2 \cdot Z^1_1 Z^2}}
\]
by repeating the reasoning that led to Equation \ref{eq:Z1a} but using Lemmas \ref{lem:nonroot3} and \ref{lem:nonroot4} instead. Moreover, combining the last equation with Equations \ref{eq:Z11} and \ref{eq:Z2a} gives
\[
\sigma_{XY \cdot Z^1_1 Z^2_2} =  \sigma_{XY} \frac{\sigma^2_{X_1 \cdot Z^1_1}}{\sigma^2_{X_1}} \frac{\sigma^2_{X_2 \cdot Z^1_1 Z^2_2}}{\sigma^2_{X_2 \cdot Z^1_1 Z^2}}.
\]

Finally, continuing with the process above for $X_3, \ldots, X_{m+n}$ yields
\[
\sigma_{X Y \cdot Z^{1:m+n}_{1:m+n}} = \sigma_{X Y} \frac{\sigma^2_{X_1 \cdot Z_1^1}}{\sigma^2_{X_1}} \prod_{i=2}^{m+n} \frac{\sigma^2_{X_i \cdot Z_{1:i}^{1:i}}}{\sigma^2_{X_i \cdot Z_{1:i-1}^{1:i}}}
\]
which implies the desired result by repeated application of Lemma \ref{lem:aux3}.
\end{proof}

\begin{proof}[Proof of Theorem \ref{the:condpathnonroot}]
The proof is analogous to that of Theorem \ref{the:condpathroot}.
\end{proof}

\bibliographystyle{plainnat}
\bibliography{MoreCondPathAnalysis}

\begin{thebibliography}{10}
\providecommand{\natexlab}[1]{#1}
\providecommand{\url}[1]{\texttt{#1}}
\expandafter\ifx\csname urlstyle\endcsname\relax
  \providecommand{\doi}[1]{doi: #1}\else
  \providecommand{\doi}{doi: \begingroup \urlstyle{rm}\Url}\fi

\bibitem[Chaudhuri(2005)]{Chaudhuri2005}
S.~Chaudhuri.
\newblock \emph{{Using the Structure of d-Connecting Paths as a Qualitative
  Measure of the Strength of Dependence}}.
\newblock PhD thesis, University of Washington, 2005.

\bibitem[Chaudhuri(2014)]{Chaudhuri2014}
S.~Chaudhuri.
\newblock {Qualitative Inequalities for Squared Partial Correlations of a
  Gaussian Random Vector}.
\newblock \emph{Annals of the Institute of Statistical Mathematics},
  66:\penalty0 345--367, 2014.

\bibitem[Chaudhuri and Richardson(2003)]{ChaudhuriandRichardson2003}
S.~Chaudhuri and T.~Richardson.
\newblock {Using the Structure of d-Connecting Paths as a Qualitative Measure
  of the Strength of Dependence}.
\newblock In \emph{Proceedings of the 19th Conference on Uncertainty in
  Artificial Intelligence}, pages 116--123, 2003.

\bibitem[Chaudhuri and Tan(2010)]{ChaudhuriandTan2010}
S.~Chaudhuri and G.~L. Tan.
\newblock {On Qualitative Comparison of Partial Regression Coefficients for
  Gaussian Graphical Markov Models}.
\newblock In \emph{Algebraic Methods in Statistics and Probability II.
  Contemporary Mathematics, Vol. 516}, pages 125--133. American Mathematical
  Society, 2010.

\bibitem[Ong(2014)]{Ong2014}
V.~M.~H. Ong.
\newblock \emph{{Model Selection for Graphical Markov Models}}.
\newblock PhD thesis, National University of Singapore, 2014.

\bibitem[Pe\~{n}a(2020)]{Penna2020}
J.~M. Pe\~{n}a.
\newblock { Conditional Path Analysis in Singly-Connected Path Diagrams}.
\newblock \emph{arXiv:2002.05226 [stat.ME]}, 2020.

\bibitem[Pearl(2009)]{Pearl2009}
J.~Pearl.
\newblock \emph{Causality: Models, Reasoning, and Inference}.
\newblock Cambridge University Press, 2009.

\bibitem[Pearl(2013)]{Pearl2013}
J.~Pearl.
\newblock {Linear Models: A Useful ``Microscope'' for Causal Analysis}.
\newblock \emph{Journal of Causal Inference}, 1:\penalty0 155--170, 2013.

\bibitem[Pearl(2014)]{Pearl2014}
J.~Pearl.
\newblock {Comment: Understanding Simpson’s Paradox}.
\newblock \emph{The American Statistician}, 68:\penalty0 8--13, 2014.

\bibitem[Wright(1921)]{Wright1921}
S.~Wright.
\newblock {Correlation and Causation}.
\newblock \emph{Journal of Agricultural Research}, 20:\penalty0 557--585, 1921.

\end{thebibliography}

\end{document}